\documentclass[11pt]{amsart}

\usepackage[all, cmtip]{xy}

\usepackage{times,amsfonts,amsmath,amstext,amsbsy,amssymb, amsopn,amsthm,upref,eucal, mathrsfs}

\newcommand \la {\lambda}

\newcommand \Conf {{\mathrm {Conf}}}

\newcommand \Prob {{\mathbb P}}

\newcommand \be {{\widetilde B}(E)}

\newcommand \pe {{\widetilde \Pi}(E)}

\newcommand \C {{\mathbb C}}
\newcommand \R {{\mathbb R}}
\newcommand \E {S}
\newcommand \im {\mathrm{ Im }}
\usepackage{color}

\newtheorem{theorem}{Theorem}

\newtheorem{corollary}[theorem]{Corollary}
\newtheorem{proposition}[theorem]{Proposition}
\theoremstyle{remark}
\newtheorem{ex}{Example}

\begin{document}

%\runninghead{Determinantal point processes associated with de Branges spaces}
\title[Determinantal point processes associated with de Branges spaces]{Quasi-symmetries and rigidity for determinantal point processes
associated with de Branges spaces}

\date{}

\author{Alexander I. Bufetov}
\address{Alexander I. Bufetov\\ Aix-Marseille Universit{\'e}, CNRS, Centrale Marseille, I2M, UMR 7373}
\address{Steklov Institute of Mathematics, Moscow}
\address{Institute for Information Transmission Problems, Moscow}
\address{National Research University Higher School of Economics, Moscow}

\author{Tomoyuki Shirai }
\address{Tomoyuki Shirai\\ Institute of Mathematics for Industry, Kyushu University, Fukuoka, 819-0395,}
\address{ Japan}

\maketitle
\begin{abstract}
In this note, we show that
 determinantal point processes on the real line corresponding
 to de Branges spaces of entire functions are rigid in the sense of
 Ghosh-Peres and, under certain additional assumptions, quasi-invariant
 under the group of diffeomorphisms of the line with compact support. 
\end{abstract}
\section{De Branges Spaces}
Recall that a  de Branges function is an entire function $E$ satisfying
\[
 |E(z)| > |E^{\#}(z)| \quad \text{for $z \in \C_+$},
\]
where $E^{\#}(z) = \overline{E(\bar{z})}$.
We note that such an entire function $E$ does not have zeros in $\C_+$.
The de Branges space associated with $E$
is a Hilbert space $B(E)$ of entire functions such that

(i) $f|_{\R} \in L^2(\R, |E(\la)|^{-2}d\la)$,
and 

(ii)$\Big| \frac{f(z)}{E(z)} \Big|, \Big|
\frac{f^{\#}(z)}{E(z)} \Big|\le C_f (\im z)^{-1/2}$ for $z \in \C_+$,

where $f|_{\R}$ is the restriction of $f$ on $\R$.
       Under the condition (i),
       the condition (ii) is equivalent to the requirement that $f/E$ and $f^{\#}/E$
       belong to the Hardy
       space $H_2$ on the upper-half plane $\C_+$. The de Branges space
       is a natural generalization of the Paley-Wiener space which is
       associated with the de Branges function $E(z) = e^{-iaz}$.

The Hilbert space $B(E)$ admits the following reproducing kernel:
\[
 \Pi(E)(z,w)
 = \frac{E(z) \overline{E(w)} - E^{\#}(z) \overline{E^{\#}(w)}}{-2\pi i
 (z - \bar{w})},
\]
i.e., for any $f \in B(E)$, we have 
\[
 f(z) = \int_{\R} \Pi(E)(z,\la)f(\la)|E(\la)|^{-2} d\la.
\]
The diagonal value is given by
\[
 \Pi(E)(z,z) = \frac{|E(z)|^2 - |E^{\#}(z)|^2}{4\pi \im z} > 0 \quad (z
 \in \C \setminus \R),
\]
and
\[
 \Pi(E)(x,x) = \frac{1}{2\pi}
 \frac{\partial}{\partial y} |E(x+iy)|^2 \Big|_{y=0}
 \quad (x \in \R).
\]
The Hilbert space $B(E)$ is naturally identified with a subspace of
$L_2({\mathbb R}, |E(\la)|^{-2}d\la)$.

It will, however, be more convenient for us to consider the space
$$
\be=\left\{\frac{F(\la)}{E(\la)}, F\in B(E)\right\},
$$
which is then naturally identified with a subspace of $L_2({\mathbb
R})$.
Let $\pe: L_2({\mathbb R}) \to \be$ be the corresponding operator of
orthogonal projection with kernel
\[
 \pe(z,w) = \Pi(E)(z,w) \big(E(z) \overline{E(w)} \big)^{-1}.
\]
In this note we study  the  determinantal point process $\Prob_{\pe}$  on
${\mathbb R}$ corresponding to
the locally trace class projection operator  $\pe$ . We recall the necessary definitions.

\section{Determinantal Point Processes}

\subsection{Locally  trace class operators and their kernels.}
Let $\mu$ be a $\sigma$-finite Borel measure on a Polish space $\E$.
Let ${\mathscr I}_{1}(\E,\mu)$ be the ideal of trace class operators
${\widetilde K}\colon L_2(\E,\mu)\to L_2(\E,\mu)$ (see e.g. volume~1 of~\cite{ReedSimon} for
the precise definition); the symbol
$||{\widetilde K}||_{{\mathscr I}_{1}}$ will stand for the
${\mathscr I}_{1}$-norm of the operator ${\widetilde K}$.

Let  $\mathscr I_{1,  \mathrm{loc}}(\E,\mu)$ be the space of operators
$K\colon L_2(\E,\mu)\to L_2(\E,\mu)$
such that for any bounded Borel subset $B\subset \E$
we have $$\chi_BK\chi_B\in{\mathscr I}_1(\E,\mu).$$
{Such an operator $K$ is called a locally trace class operator.}
We endow the space ${\mathscr I}_{1, \mathrm{loc}}(\E,\mu)$
with a countable family of semi-norms
\begin{equation}
\label{btrcl}
||\chi_BK\chi_B||_{{\mathscr I}_1}
\end{equation}
where $B$ runs through an exhausting family $B_n$ of bounded sets.
A locally trace class operator $K$ admits a {\it kernel}, for which, slightly abusing notation, we use the same symbol $K$.

\subsection{Determinantal Point Processes}

A Borel probability measure $\mathbb{P}$ on
$\Conf(\E)${, the space of locally finite configurations,} is called
\textit{determinantal} if there exists an operator
$K\in{\mathscr I}_{1,  \mathrm{loc}}(\E,\mu)$ such that for any bounded measurable
function $g$, for which $g-1$ is supported in a bounded set $B$,
we have
\begin{equation}
\label{eq1}
\mathbb{E}_{\mathbb{P}}\Psi_g
=\det\biggl(1+(g-1)K\chi_{B}\biggr),
\end{equation}
{where $\Psi_g(X) = \prod\limits_{x \in X} g(x)$ for $X\in \Conf(S)$.
}
The Fredholm determinant in~\eqref{eq1} is well-defined since
$K\in {\mathscr I}_{1, \mathrm{loc}}(E,\mu)$.
The equation (\ref{eq1}) determines the measure $\Prob$ uniquely.

For any
pairwise disjoint bounded Borel sets $B_1,\dotsc,B_l\subset \E$
and any  $z_1,\dotsc,z_l\in {\mathbb C}$ from (\ref{eq1}) we  have
$$\mathbb{E}_{\mathbb{P}}z_1^{\#_{B_1}}\dotsb z_l^{\#_{B_l}}
=\det\biggl(1+\sum\limits_{j=1}^l(z_j-1)\chi_{B_j}K\chi_{\sqcup_i B_i}\biggr).$$

If $K$ belongs to
${\mathscr I}_{1, \text{loc}}(\E,\mu)$, then, throughout the paper, we denote
the corresponding determinantal measure by
$\mathbb{P}_K$.
If $K\in {\mathscr I}_{1, \text{loc}}(\E,\mu)$,
then the existence of the probability measure $\Prob_K$
is guaranteed (\cite{ShirTaka0}, \cite{Soshnikov}).
For further results and background on determinantal point processes, see
e.g. \cite{ghosh}, \cite{HoughEtAl}, \cite{Lyons},
   \cite{ShirTaka1}, \cite{ShirTaka2}, \cite{Soshnikov}.

 \section{The integrable form of the reproducing kernel}
 Our aim in this note is to study rigidity (in the sense of Ghosh and
 Peres) and the quasi-symmetries of the point process $\Prob_{\pe}$. We
 start by fixing some notation. 
For a de Branges function $E$,
we set
$$
A(z)=\frac{E(z)+E^{\#}(z)}{2}, \ B(z)=\frac{E(z)-E^{\#}(z)}{2i}.
$$
The kernel of the operator $\pe$, essentially the reproducing kernel of
our de Branges space, takes the form
$$
\pe(x,y)= \frac{1}{\pi} \frac{A(x)B(y)-B(x)A(y)}{ (x-y) E(x)
\overline{E(y)}}, \ x,y\in {\mathbb R}.
$$
Slightly abusing notation, we keep for the kernel the same symbol as for the operator.
For the diagonal values,
{it is easy to see that}
\begin{align}
\pe(x,x)
&= \frac{1}{2\pi} |E(x)|^{-2}\frac{\partial}{\partial y}
|E(x+iy)|^2\Big|_{y=0} \label{eq:pitilde}\\
&= \frac{1}{\pi} \frac{\partial}{\partial y} \log |E(x+iy)|
 \Big|_{y=0}. \nonumber
\end{align}

The kernel $\pe$ has {\it integrable} form.
Corollary 2.2 in \cite{buf-rig} now implies the rigidity, in the sense
of Ghosh and Peres \cite{ghosh}, \cite{GP}, of the determinantal measure
$\Prob_{\pe}$.
Before giving the notion of rigidity and our results, we provide some examples of
DPPs.

\section{Examples of determinantal point processes associated with de Branges spaces}
Here we give some examples of determinantal point process (DPP) associated
with de Branges space.
\begin{ex}[A class of orthogonal polynomial ensembles] 
 Let $E(z) = \prod_{i=1}^n (z+a_i)$ for $a_i \in \C_{+}$.
In this case, $B(E)$ is the space of polynomials of degree less than or equal to
 $n-1$. The corresponding DPP is the $n$-th orthogonal polynomial ensemble
 with weight $|E(\la)|^{-2}$.
 In particular, its intensity is given by
\[
\pe(x,x) = \frac{1}{\pi} \sum_{i=1}^{n} \frac{\im a_i}{|x+a_i|^2}.
\]
\end{ex}
\begin{ex}[Sine-process]
The Paley-Wiener space, for which $E(z) = e^{-iaz} \ (a>0)$,
 $A(z) = \cos az$, $B(z) = -\sin az$ yields the sine-kernel $\pe(x,y) = \frac{\sin
 a(x-y)}{\pi(x-y)}$.
\end{ex}

\begin{ex}[Eigenfunction expansion for Schr\"odinger equation]
Fix $\ell \in (0,\infty]$. For $V \in L^1_{\text{loc}}([0,\ell))$,
we consider the Schr\"odinger equation
\[
- \varphi_{\la}'' + V \varphi_{\la} = \la \varphi_{\la} \quad (\la \in \C)
\]
with $\varphi_{\la}(0)=1$ and $\varphi'_{\la}(0)=0$.
The solution $\varphi_{\la}(x)$ is jointly continuous in $(\la, x)$ and
entire in $\la$.
Suppose that the right boundary $x=\ell$ is of the limit circle type.
Then, for each fixed $b \in (0, \ell)$,
\[
 E_b(z) = \varphi_z(b) + i \varphi'_{z}(b),
\]
 defines a de Branges function. In this case,
\begin{align*}
 \Pi(E_{b})(z,w)
&= \frac{1}{\pi}
 \frac{\varphi_z(b) \overline{\varphi'_w(b)} - \varphi'_z(b) \overline{\varphi_w(b)}}{z-\bar{w}} \\
&= \frac{1}{\pi} \int_0^{b} \varphi_{z}(t)
 \overline{\varphi_{w}(t)} dt.
\end{align*}
The intensity of the corresponding DPP is given by
\begin{align*}
\widetilde{\Pi}(E_{b})(\la,\la)
= \frac{1}{\pi}
\frac{\int_0^{b} |\varphi_{\la}(t)|^2 dt}{|\varphi_{\la}(b)|^2 + |\varphi'_{\la}(b)|^2}
\end{align*}
\end{ex}

\section{Ghosh-Peres Rigidity.}
Given a bounded subset $B \subset {\mathbb R}$ and a configuration $X
\in \text{Conf}(\R),$ let $ \#_B(X)$ stand for the number of particles of
$X$ lying in $B$. Given a Borel subset $C \subset {\mathbb R}$, we let
$\mathcal{F}_C$ be the $\sigma$-algebra generated by all random
variables of the form $\#_B, B\subset C.$ If $\mathbb{P}$ is a point
process on ${\mathbb R}$ then we write $\mathcal{F}_C^{\mathbb{P}}$ for
the $\mathbb{P}$-completion of $\mathcal{F}_C$.

{\bf Definition (Ghosh and Peres \cite{ghosh}, \cite{GP}).} A point
process $\mathbb{P}$  is called {\bf rigid} if for any bounded Borel
subset $B$ the random variable $\#_B$ is  $\mathcal{F}_{{\mathbb
R}\backslash B}^{\mathbb{P}}$-measurable.

\begin{theorem}
The determinantal measure $\Prob_{\pe}$ is rigid in the sense of Ghosh
 and Peres.
\end{theorem}
\begin{proof}
By Corollary 2.2 in \cite{buf-rig}, we need to establish the
existence of $R>0$, $C>0$  and $\varepsilon>0$ such that
 for all $|x|<R$ we have $|A(x)|\leq
      C|x|^{-1/2+\varepsilon}|E(x)|$;  $|B(x)|\leq
      C|x|^{-1/2+\varepsilon}|E(x)|$ and for all $|x|>R$ we have $|A(x)|\leq C|x|^{1/2-\varepsilon}|E(x)|$;
$|B(x)|\leq C|x|^{1/2-\varepsilon}|E(x)|$;
and these conditions  hold since  $|A(x)|, |B(x)|\leq |E(x)|$.
\end{proof}

Proposition 8.1 in \cite{BQ} now implies the following
\begin{corollary}\label{rigid-disjoint}
For any $k,l\in {\mathbb N}$, $k\neq l$, for almost any $k$-tuple $(p_1,
 \dots, p_k)$ and almost any $l$-tuple $(q_1, \dots, q_l)$ of distinct
 points in $\mathbb R$, the reduced Palm measures $\Prob_{\pe}^{p_1,
 \dots, p_k}$ and $\Prob_{\pe}^{q_1, \dots, q_l}$ are mutually
 singular.
\end{corollary}

\section{Quasi-Symmetries}

We next give sufficient conditions for the equivalence of Palm measures
of the same order.

Let $p_1, \ldots, p_l, q_1, \ldots, q_l \in \mathbb{R}$ be distinct.
For $R>0$, $\varepsilon>0$ and a configuration $X$ on $\mathbb{R}$  write
$${{\overline \Psi}_{R, \varepsilon}(p_1, \ldots, p_l; q_1,
 \ldots, q_l; X)}=  C(R, \varepsilon)\times\prod_{x \in X, |x| \leq R, \min |x-q_i|\geq
 \varepsilon} \prod_{i=1}^l \left( \displaystyle \frac{x-p_i}{x-q_i}
 \right)^2,
$$
where the constant $C(R, \varepsilon)$ is chosen in such a way that
\begin{equation}
\displaystyle \int\limits_{\Conf({\mathbb R})}
{\overline \Psi}_{R, \varepsilon}(p_1, \ldots, p_l; q_1, \ldots, q_l; X)
d\Prob_{\pe}^{q_1, \dots, q_l}=1.
\end{equation}

We will often need the following assumption on our de Branges function
$E$:
\begin{equation} \label{main-assum}
\int\limits_{\mathbb R} \displaystyle\frac{\frac{\partial}{\partial y}
 |E(x+iy)|^2|_{y=0}}{(1+x^2)|E(x)|^{2}} dx <+\infty.
\end{equation}

  Given our de Branges function $E$, there exists a nondecreasing
  continuous function $\phi$ on $\R$ such that
       $E(x) \exp(i \phi(x))$ is real for all $x \in \R$.
The function $\phi(x)$ is called a \textit{phase function} associated with
       $E(z)$.
       We note that
       \begin{equation}
	\phi'(x) = \pi \pe(x,x) > 0 \quad (\forall x \in \R).
	 \label{phiprime}
       \end{equation}
       (See de Branges \cite{deBr} Problem 48.)
From \eqref{eq:pitilde} and \eqref{phiprime},
the assumption (\ref{main-assum})  can equivalently be reformulated as follows
\begin{equation}
\int_{\R} \frac{\phi'(x)}{1+x^2} dx
= \int_{\R} \frac{d \phi(x)}{1+x^2} < \infty.
\label{eq:phix}
\end{equation}
It is known that there exists a $p>0$ such that
\begin{equation}\label{eq:logder}
{\displaystyle \frac{\partial}{\partial y} \log |E(x+iy)|}  
= p y + \displaystyle\frac{1}{\pi} \displaystyle\int\limits_{-\infty}^{\infty} \displaystyle \frac{y}{(t-x)^2 + y^2}
 d\phi(t) \nonumber
\end{equation}
if $E$ has no real zeros
and $|E(x+iy)|$ is a nondecreasing function of $y>0$ for each $x \in
\R$.
(See de Branges \cite{deBr} Problem 63.)

 {\bf Remark.}
  If $E$ is of exponential type and has no real zeros, then
  the condition \eqref{eq:phix} holds.
  Indeed, if $E$ is of exponential type, then
$|E(x+iy)|$ is nondecreasing in $y>0$ (See Dym \cite{Dym} Lemma 4.1).
  Setting $x=0$ and $y=1$ in \eqref{eq:logder} yields \eqref{eq:phix}.
  In particular, if $E$ is \textit{short} in the sense that $B(E)$ is closed
  under the map $f(z) \mapsto \frac{f(z) - f(i)}{z-i}$
  (see \cite{DymMcKean} Proposition 6.2.2),
  then \eqref{eq:phix} holds.

\begin{proposition}\label{regmult-cont}
Let $E$ be a de Branges function satisfying \eqref{eq:phix}.
 Then the limit
 $$
  {{\overline \Psi}(p_1, \ldots, p_l; q_1, \ldots, q_l; X)}
  = \lim\limits_{R
 \to \infty, \varepsilon\to 0} \Psi_{R, \varepsilon} (p_1, \ldots, p_l;
 q_1, \ldots, q_l; X)
$$
exists  in $L_1(\Conf({\mathbb R}), \Prob_{\pe}^{q_1, \dots, q_l})$ as well as
almost surely along a subsequence,  and satisfies
\begin{equation}
\displaystyle \int\limits_{\Conf({\mathbb R})}
{\overline \Psi}(p_1, \ldots, p_l; q_1, \ldots, q_l; X)
d\Prob_{\pe}^{q_1, \dots, q_l}=1.
\end{equation}
\end{proposition}

Corollary 4.12 in \cite{buf-gibbs} now directly implies

\begin{proposition} \label{palm-meas-cont} Let $E$ be a de Branges
 function satisfying \eqref{eq:phix}.
Then for any distinct points
 $p_1, \dots, p_l$, ${q}_1,\dots, {q}_l\in {\mathbb R}$,  the
 corresponding reduced Palm measures are equivalent, and we have
$$
\displaystyle \displaystyle \frac{{d{\mathbb P}_{\Pi^{{p}_1,\dots, {
 p}_l}}}}{{d{{\mathbb P}_{\Pi^{{q}_1,\dots, {q}_l}}}}}(X)={\overline
 \Psi}(p_1, \ldots, p_l; q_1, \ldots, q_l; X).
$$
\end{proposition}

Theorem 1.5 in \cite{buf-gibbs}  directly implies the following
\begin{proposition}\label{QS}
Let $E$ be a de Branges function satisfying \eqref{eq:phix}.
Let $F: {\mathbb R}\to {\mathbb R}$ be a  diffeomorphism acting as the
 identity beyond a bounded open set $V\subset {\mathbb R}$. For
 $\mathbb{P}_{\pe}$-almost every configuration $X\in \Conf({\mathbb R})$
 the following holds.
If $X\bigcap V=\{q_1, \dots, q_l\}$, then
\begin{multline}\label{RN}
\displaystyle \frac{d\mathbb{P}_{\pe}\circ
 F}{d\mathbb{P}_{\pe}}(X)={\overline \Psi}(F(q_1), \dots, F(q_l); q_1,
 \dots, q_l; X)\times
 \\ \times \displaystyle \frac{\det({ \pe}(F(q_i),
 F(q_j))_{i,j=1,\dots,l}}{\det({\pe}(q_i,q_j))_{i,j=1,\dots,l}}
 \times \displaystyle F^{\prime} (q_1)\ldots \displaystyle
 F^{\prime}(q_l).
 \end{multline}
\end{proposition}
{\bf Remark.} The open set $V$ can be chosen in many ways; the resulting
value of the Radon-Nikodym derivative is of course the same.

{\bf Remark.} As in \cite{buf-gibbs}, $F$ can, more generally, be a
compactly supported Borel automorphism preserving the Lebesgue measure
class. In this case, the derivative $F^{\prime}$ in (\ref{RN}) should be
replaced by the Radon-Nikodym derivative of the Lebesgue measure under
$F$.

{\bf Remark.} Conditional measures of our DPPs can now also be found using
the results of \cite{buf-cond}.

{\bf Acknowledgements.}
We are grateful to Dmitri Chelkak for useful discussions. 
The research of A. Bufetov on this project has received funding from the European Research Council (ERC) under the European Union's Horizon 2020 research and innovation programme under grant agreement No 647133 (ICHAOS). 
It was also supported  by a subsidy granted to the HSE by the Government
of the Russian Federation for the implementation of the Global
Competitiveness Programme. The research of T.~Shirai was supported in
part by the Japan Society for the Promotion of Science (JSPS) Grant-in-Aid for Scientific Research (B) 26287019.

\end{document}